\documentclass[12pt]{amsart}
\usepackage[utf8]{inputenc}
\usepackage{amsmath}
\usepackage{amssymb}
\usepackage{amsfonts}
\usepackage{amsthm}
\usepackage{mathrsfs} 
\usepackage{bm}
\usepackage{bbm}
\usepackage{tikz}
\usepackage{centernot}
\usepackage{hyperref}
\usepackage[ruled,vlined]{algorithm2e}
\usepackage[margin=1.2in]{geometry}
\usepackage[shortlabels]{enumitem}
\numberwithin{equation}{section}

\DeclareMathOperator{\im}{Im}

\DeclareMathOperator{\Bin}{Bin}

\newcommand{\E}{\Bbb{E}}

\renewcommand{\P}{\Bbb{P}}

\newcommand{\CC}{\mathcal{C}}

\newcommand{\EE}{\mathcal{E}}

\newcommand{\UU}{\mathcal{U}}

\newcommand{\ep}{\epsilon}

\hypersetup{
    colorlinks=true,
    linkcolor=blue,
    filecolor=magenta,
    urlcolor=blue,
    citecolor=blue
}

\makeatletter
\newtheorem*{rep@theorem}{\rep@title}
\newcommand{\newreptheorem}[2]{%
\newenvironment{rep#1}[1]{%
 \def\rep@title{#2 \ref{##1}}%
 \begin{rep@theorem}}%
 {\end{rep@theorem}}}
\makeatother

\newtheorem{thm}{Theorem}
\newreptheorem{thm}{Theorem}

\newtheorem{result}{Result}[section]
\newtheorem{lem}[result]{Lemma}
\newtheorem{prp}[result]{Proposition}

\newtheorem{clm}[result]{Claim}

\theoremstyle{definition}
\newtheorem{rmk}[result]{Remark}
\newtheorem*{defn}{Definition}

\newtheorem{que}{Question}

\newtheorem*{ack}{Acknowledgements}

\theoremstyle{remark}

\DeclareMathOperator{\sseq}{\bm{\mathsf{ss}}}

\DeclareMathOperator{\LCS}{\bm{\mathsf{LCS}}}
\DeclareMathOperator{\strng}{\bm{\mathsf{string}}}
\DeclareMathOperator{\weak}{\bm{\mathsf{weak}}}

\newcommand{\hide}[1]{}
\newcommand{\edit}[1]{}

\newcommand{\rough}[1]{}
\definecolor{darkgreen}{RGB}{75,150,75}
\newcommand{\review}[1]{}
\newcommand{\dc}[1]{\textcolor{orange}{dc: #1}}

\newcommand{\hides}[1]{}
\newcommand{\pub}[1]{}

\title{A notion of twins}
\author{Zach Hunter}
\email{zachary.hunter@exeter.ox.ac.uk}
\date{\today}

\begin{document}

\maketitle

\begin{abstract}Given a combinatorial structure, a ``twin'' is a pair of disjoint substructures which are isomorphic (or look the same in some sense). In recent years, there have been many problems about finding large twins in various combinatorial structures. For example, given a graph $G$, one can ask what is the largest $s$ such that there exist disjoint subsets $I,J\subset V(G)$ on $s$ vertices, such that the induced subgraphs $G[I],G[J]$ are isomorphic.

We are motivated by two different problems of finding twins in two kinds of ordered objects (strings and permutations). We introduce a new variant of ``twin problem'' which generalizes both of these. By considering this generalization, we are able to improve some bounds obtained by Dudek, Grytczuk, and Ruci\'nski,  and give a negative answer to a conjecture of theirs.
\end{abstract}

\section{Introduction}

For positive integer $n$, we write $[n]:=\{1,\dots,n\}$. Given a set $X$ and positive integer $s$, we write $X^{(s)}:= \{S\subset X:|S|=s\}$.

\subsection{Recent work in ordered settings}

We start by recalling two problems about finding ``twins'' in various ordered objects. A generalization of these shall be the focus of this note. 

Firstly, given a binary string $x\in \{0,1\}^n$, a \textit{string-twin} is a pair of disjoint indices $I =\{i_1<\dots<i_\ell\},J = \{j_1<\dots<j_\ell\} \subset [n]$, such that the subsequences $x|_I$ and $x|_J$ are equal (formally, that $x(i_t) = x(j_t)$ for each $t=1,\dots,\ell$). We define $f^{\strng}(x)$ to be the maximum $\ell$ such that there exists a string-twin $I,J$ of $x$ with $|I|=\ell$. Since $I,J$ are disjoint subsets of $n$ of equal length, it is obvious that $f^{\strng}(x)\le n/2$ always holds. 

In the work of Axenovich, Person, and Puzynina, it was shown that this upper bound was asymptotically tight (i.e., that every binary string $x\in \{0,1\}^n$ had a twin of length $(1/2-o(1))n$) \cite{axenovich}. This was acheived by establishing a celebrated ``regularity lemma for strings''. 

One can also ask analogous questions about the length of twins in $r$-ary strings $x\in [r]^n$ (here, we say $I,J\subset [n]$ are string-twins of $x$ if they are disjoint and $x|_I = x|_J$). Write $F_r^{\strng}(n)$ to denote the minimum of $f^{\strng}(x)$ over all $x\in [r]^n$. A trivial consequence of \cite{axenovich} tells us that $F_r^{\strng}(n) \ge (1/r-o(1))n$ (this is seen by passing to the subsequence induced on the two most popular letters of any $r$-ary string, and then applying their bound for the binary case). This later was further improved by Bukh and Zhou, who prove that $F_r^{\strng}(n) = \Omega(n/r^{2/3})$ \cite{bukh}. 

Recently, Dudek, Grytczuk, and Ruci\'nski introduced a similar problem involving permutations. Given a permutation $\pi\in S_n$, we define its sign sequence $\sseq(\pi) \in \{-1,1\}^{n-1}$ so that $\sseq(\pi)_i =1$ if and only if $\pi(i)< \pi(i+1)$. We then define a \textit{weak-twin} (of $\pi$) to be a disjoint pair of indices $I = \{i_1<\dots<i_\ell\}, J = \{j_1<\dots<j_\ell\}\subset [n]$, such that $\sseq(\pi|_I) = \sseq(\pi|_J)$ (formally, that $\pi(i_t)<\pi(i_{t+1})$ if and only if $\pi(j_t)<\pi(j_{t+1})$ for $t\in [\ell-1]$).

Let $f^{\weak}(\pi)$ denote the largest $\ell$ such that there exists a weak-twin $I,J$ of $\pi$ with $|I| = \ell$. And write $F^{\weak}(n)$ to denote the minimum of $f^{\weak}(\pi)$ over all $\pi\in S_n$.

Dudek et al. proved that every $\pi \in S_n$ has a weak-twin of length $\ge n/12$ \cite{dudek} (i.e., that $F^{\weak}(n) \ge n/12$). They furthermore conjectured that, like in the binary string case, every $\pi \in S_n$ should have a weak-twin of length $(1/2-o(1))n$. We shall improve upon both their upper and lower bounds for $F^{\weak}(n)$.

\subsection{A general twin problem}

\hide{For positive integer $n$, we write $[n]:=\{1,\dots,n\}$.} We write $K_n$ to denote the complete graph on vertex set $[n]$ (and here, the labels of vertices will be important).

Given positive integers $n,r$, we write $\CC_{n;r}$ to denote the set of $r$-edge-colorings of $K_n$ (i.e., the set of $c:[n]^{(2)}\to [r]$). We shall also write $\CC_n$ to denote $\bigcup_{r=1}^\infty \CC_{n;r}$ the set of all finite colorings.

Now given $c\in \CC_{n}$, a \textit{twin of size $\ell$} (with respect to $c$) is a pair of subsets $I=\{i_1<\dots<i_\ell\},J=\{j_1<\dots<j_\ell\}\subset [n]$, such that:
\begin{itemize}
    \item $I,J$ are disjoint;
    \item $c(\{i_t,i_{t+1}\}) = c(\{j_t,j_{t+1}\})$ for $t=1,\dots,\ell-1$.
\end{itemize}
Given a colored ordered object $c\in \CC$, we define $f(c)$ to be maximum $\ell$ where there exists a twin (wrt $c$) of size $\ell$. For $r,n$, we define $F_r(n)$ to be the minimum of $f(c)$ over all $c\in \CC_{n;r}$. 

\hide{We analogously define $f^{\strng}(x)$ on $r$-ary strings $x\in [r]^n$ and $f^{\weak}(\pi)$ on permutations $\pi \in S_n$. Furthermore, we define $F_r^{\strng}(n)$ to be the minimum of $f^{\strng}(x)$ over all $x\in [r]^n$ and $F^{\weak}(n)$ to be the minimum of $f^{\weak}(\pi)$ over all $\pi \in S_n$.}

While the notions of string-twins and weak-twins are somewhat similar in appearance, there doesn't seem to be a direct way to encode instances from one setting into the other (e.g., given a binary string $x\in \{0,1\}^n$, we don't know how to create some $\pi_x\in S_n$ where the string-twins of $x$ and weak-twins of $\pi_x$ are at all related). However, both settings can be encoded by our notion of twins. Specifically, we will establish the two following reductions in Section~\ref{reductions}.
\begin{prp}\label{general to weak}We have that
\[F_2(n)\le F^{\weak}(n).\]
\end{prp}
\begin{prp}\label{general to string}We have that
\[F_r(n)+1\le F_r^{\strng}(n).\]
\end{prp}

For lower bounds, we prove the following in Section~\ref{lower bound}.
\begin{thm}\label{main}We have
\[F_r(n) \ge n/(r^2+1)-O_r(1).\]
\end{thm}
\noindent By Proposition~\ref{general to weak}, Theorem~\ref{main} tells us $F^{\weak}(n)\ge F_2(n) \ge n/5-O(1)$, improving Dudek et al.'s previous bound of $n/12-O(1)$. But with two colors we can do even better, in Section~\ref{improvement} we establish the following.
\begin{thm}\label{2better}We have
\[F_2(n)\ge \frac{1}{4}n-O(1).\]
Consequently, $F^{\weak}(n) \ge \frac{1}{4}n-O(1)$.
\end{thm}

We also establish some upper bounds.
\begin{thm}\label{general upper}We have 
\[F_r(n) = O(n/r).\]
\end{thm}\noindent As previously noted, Bukh and Zhou proved that $F_r^{\strng} = \Omega(n/r^{2/3})$ which is greater than $F_r(n)$ for large $n$ (assuming $r$ is big enough).

Furthermore, we can adapt our techniques to establish the following.
\begin{thm}\label{weak upper}There exists some absolute constant $\eta>0$, so that for all large $n$ we have 
\[(1/2-\eta)n\ge F^{\weak}(n)\ge F_2(n).\]
\end{thm}\noindent This contradicts a conjecture of Dudek et al.

\subsection{Organization}

In Section~\ref{reductions}, we briefly deduce Propositions~\ref{general to weak} and \ref{general to string}, which demonstrate that our twin problem appropriately ``encodes'' the notions of string-twins and weak-twins. In Section~\ref{lower bound}, we prove our general lower bound for $F_r(n)$ (Theorem~\ref{main}). In Section~\ref{improvement}, we get an improved lower bound for $F_2(n)$ (Theorem~\ref{2better}). In Section~\ref{upper bounds} we get our two upper bounds (Theorem~\ref{general upper} and Theorem~\ref{weak upper}).

\begin{ack} We thank Matthew Kwan for introducing us to the weak-twin problem posed in \cite{dudek}. We also thank Daniel Carter and Benny Sudakov for helpful feedback on the writing of this manuscript.

Part of this work was conducted while the author was staying at IST Austria, we are grateful for their hospitality.

\end{ack}

\section{Reductions}\label{reductions}

In this section, we quickly establish Propositions~\ref{general to weak} and \ref{general to string}.

Proposition~\ref{general to weak} is a corollary of our first lemma.
\begin{lem}\label{weak reduction} Given $\pi \in S_n$, and there exists $c =c_\pi \in \CC_{n;2}$ so that $I,J\subset [n]$ is a twin of $c$ if and only if $I,J$ is a weak-twin of $\pi$.

\begin{proof}We define $c=c_\pi: [n]^{(2)}\to \{-1,1\}$, which we consider an element of $\CC_{n;2}$ (since nothing really changes upon relabelling our color palette). For $e=\{i<j\}\in [n]^{(2)}$, we set $c(e) = 1$ if and only if $\pi(i)<\pi(j)$. For $S= \{s_1<\dots<s_\ell\}\subset [n]$, we have that $c(\{s_1,s_2\}),\dots,c(\{s_{\ell-1},s_\ell\}) = \sseq(\pi|_S)$. 

Thus, $I,J\subset [n]$ are twins (wrt $c$) if and only if $I,J$ are disjoint and $\sseq(\pi|_I) = \pi(\pi|_J)$ (which is equivalent to $I,J$ being weak-twins (of $\pi$)).
\end{proof} 
\end{lem}

Proposition~\ref{general to string} is a corollary of our second lemma.
\begin{lem}\label{string reduction} Given $x\in [r]^n$, there exists $c=c_x\in \CC_{n;r}$ so that $I,J$ is a twin of $c$ if and only if $I\setminus \{\max(I)\},J\setminus \{\max(J)\}$ is a string-twin of $x$ (and $|I|= |J|$).
\begin{proof}
    We define $c=c_x\in \CC_{n;r}$ as follows. For $e=\{i<j\}\in [n]^{(2)}$, we take $c(e) = x(i)$. 
    
    By construction, it is rather clear that our claim about the twins of $c$ holds.\end{proof}

\end{lem}

\section{Lower bound}\label{lower bound}

\subsection{Brief outline}\label{outline}

We are loosely motivated by the following simple proof that $F_2^{\strng}(n)\ge \lfloor n/3\rfloor$. Fix some binary string $x\in \{0,1\}^n$, and set $\ell = \lfloor n/3\rfloor $. 

For $i=1,\dots,\ell$, we have the discrete interval $E_t = \{3t-2,3t-1,3t\}\subset [n]$. By pigeonhole, we can find distinct $i_t,j_t\in E_t$ such that $x(i_t) = x(j_t)$. 

Then, we obtain a string-twin of length $\ell$ by taking $I= \{i_1<i_2<\dots <i_\ell\}$ and $J = \{j_1<j_2<\dots<j_\ell\}$. Indeed, it is clear that $I,J$ are disjoint, as $i_t,j_t\in E_t$ are always distinct and the intervals $E_t$ are disjoint. Meanwhile, since $\max(E_t)<\min(E_{t+1})$ it clear that $i_t<i_{t+1}$ and $j_t<j_{t+1}$ for $t=1,\dots,\ell-1$ (meaning the order of the indices is correct). Finally, we have $x(i_t) = x(j_t)$ for all $t\in [\ell]$, whence $x|_I = x|_J$ as desired.

Unfortunately, in our more generalized setting, such a strategy cannot work. Indeed, if we pick some $i_1,j_1\in [n]$ to be the first indices of $I$ and $J$, then it might be impossible to find $i_2>i_1,j_2>j_1$ such that $c(\{i_1,i_2\}) = c(\{j_1,j_2\})$ (e.g., the coloring could make $c(\{i_1,i_2\})$ always red and $c(\{j_1,j_2\})$ always blue). So, instead of building one twin iteratively, we shall build multiple twins. 

\subsection{Proof of Theorem~1}

Given $c \in \CC_{n;r}$, we say a pair of tuples $x,y \in [n]^2$ form a \textit{$c$-matching} if $c(\{x_1,y_1\}) = c(\{x_2,y_2\})$. 

We say a pair of $2$-sets $u,v \in [n]^{(2)}$ are \textit{$c$-matchable} if we can order these $2$-sets to get a pair of tuples which form a $c$-matching.\edit{phrasing is stupid lol}

The relevance of these definitions is the following.

\begin{lem}\label{extension}Consider some coloring $c\in \CC_n$. Let $I,J\subset [n]$ be a twin (wrt $c$) of length $\ell>0$. 

Set $u = \{\max(I),\max(J)\} \in [n]^{(2)}$. If there exists $v \in [n]^{(2)}$ where $\max(u)<\min(v)$ and $u,v$ are $c$-matchable, then there is a twin  $I',J'$ of length $\ell+1$ with $\{\max(I'),\max(J')\} = v$.
\begin{proof} By definition of $c$-matchability, we may write $v= \{i',j'\}$ so that $c(\{\max(I),i'\}) = c(\{\max(J),j'\})$.

We simply take $I' = I\cup \{i'\},J' = J\cup\{j'\}$. It is routine to check that the desired properties are satisfied. For completeness, this is done below. 

Since $\min(v)>\max(u) = \max\{\max(I),\max(J)\}$, we have that $v$ is disjoint from $I\cup J$, whence $|I'| = \ell+1 = |J'|$. Furthermore, since $v$ is a $2$-set, $i'\neq j'$, so it follows that $I',J'$ are disjoint (since by definition $I,J$ are disjoint).

Then, write $I' = \{i_1<\dots<i_{\ell+1}\},J' = \{j_1<\dots<j_{\ell+1}\}$. By assumption, we have that \[c(\{i_\ell,i_{\ell+1}\}) =c(\{\max(I),i'\}) = c(\{\max(J),j'\}) = c(\{j_\ell,j_{\ell+1}\}).\] Meanwhile, for $t\in [\ell-1]$, the fact that $I,J$ is a twin guaruntees that $c(\{i_t,i_{t+1}\}) = c(\{j_t,j_{t+1}\})$ as desired.

Hence, $I',J'$ are twins. Also, as $\min(v) > \max(I\cup J)$, it is clear that $\max(I') = i',\max(J') = j'$ implying the last property.\end{proof}
\end{lem}

We now establish a matchability result, which will allow us to carry out a modified version of the argument sketched in Subsection~\ref{outline}.

\begin{lem}\label{matchable} Let $r,k \ge 1$. Take $\Gamma = (A,B,E)$ to be a copy of $K_{r+1,rk+1}$, with $|A| = r+1,|B|= rk+1$. Consider any $r$-coloring $c:E\to [r]$.

Then there exists $B'\subset B$ with $|B'| \ge k+1$ such that for $u\in B'^{(2)}$, there exists $v\in A$ such that $u,v$ are $c$-matchable.
\begin{rmk}\label{extremal} This sharp in two aspects. First, if $|A| \le r$, then fixing some injection $\iota:A \to [r]$, then taking $c(ab) = c(\iota(a))$ will lack any $u\in B$ which are matchable.

Secondly, if $|B|<rk+1$, then we can partition $B$ into $r$ parts $B_1,\dots,B_r$ each of size at most $k$. Here if we take $c(ab) = i$ for all $b\in B_i$, we see that if $u\in B^{(2)}$ is matchable, then $u\in B_i^{(2)}$ for some $i$.\end{rmk}
\begin{proof}For $b\in B$ and $i\in [r]$, we say $b$ is $i$-popular if there are distinct $a,a'\in A$ with $c(ab) = c(a'b) = i$. By pigeonhole, we may define a map $\phi:B\to [r]$ such that $b$ is $\phi(b)$-popular for each $b\in B$ (this is because $|A|>[r]$). 

Applying pigeonhole again, there must exist some $i\in [r]$ such that $B':= \phi^{-1}(i)$ has $|B'|\ge \frac{|B|}{r}>k$ (implying $|B'|\ge k+1$ as the cardinality must be an integer). 

Finally, we note that each $\{b_1,b_2\}\in B'^{(2)}$ is $c$-matchable. Indeed, since they are both $i$-popular, there exist $a_1\in A$ and distinct $a_2,a_2'\in A$ such that \[c(a_1b_1) = c(a_2b_2) = c(a_2'b_2) = i.\]As $a_2,a_2'$ are distinct, we may WLOG assume $a_2\neq a_1$, whence $a_1b_1,a_2b_2$ is a $c$-matching.
\end{proof}
\end{lem}

\begin{proof}[Proof of Theorem~\ref{main}] Fix any $c\in \CC_{n;r}$, and set $\ell = \lfloor n/(r^2+1)\rfloor$. For $t=1,\dots,\ell$, we have the discrete interval $E_t =\{t(r^2+1)-r^2,t(r^2+1)-r^2+1,\dots,t(r^2+1)\}$.

We will now proceed by induction to find an $(r+1)$-set $U_t\in E_t^{(r+1)}$ such that for $u\in U_t^{(2)}$, there exists a twin $I,J$ of length $t$ with $\{\max(I),\max(J)\} = u$.

For $t=1$, we may take $U_1 = [r+1]$, as any pair of distinct singletons is a twin.

Then for $t=2,\dots,\ell$, we can invoke Lemma~\ref{matchable} with $A = U_{t-1},B= E_t$ to find $B'\in E_t^{(r+1)}$ satisfying the conditions of the lemma. We shall simply take $U_t = B'$. To check $U_t$ satisfies our inductive assumptions, it suffices to consider $u\in U_t^{(2)}$ and confirm there is some twin $I',J'$ of length $t$ with $\{\max(I'),\max(J')\} = u$.

By definition of $B'$, there exists $v\in U_{t-1}^{(2)}$ such that $u,v$ are $c$-matchable. By our assumptions on $U_{t-1}$, there must exists a twin $I,J$ of length $t-1$ with $\{\max(I),\max(J)\} = v$. Since $u,v$ are $c$-matchable and $\min(u)\ge \min(E_t)>\max(E_{t-1})\ge  \max(v)$, we can apply Lemma~\ref{extension} to deduce that there is in fact a twin $I',J'$ of length $t$ with $\{\max(I'),\max(J')\} = u$.
 
So, we see the induction goes through for all $t\in [\ell]$. Consequently, we see that $f(c)\ge \ell = \lfloor n/(r^2+1)\rfloor = n/(r^2+1)-O_r(1)$ as desired.\end{proof}

\section{Doing better with two colors}\label{improvement}

Here we provide a slightly ad hoc argument that improves our bound for $F_2(n)$. The idea is that the conclusion of Lemma~\ref{matchable} (with $r=k=2$) should still hold when we delete an edge from $K_{3,5}$. Meanwhile in the proof of Theorem~\ref{main}, we don't need to take the sets $U_1,\dots,U_\ell$ to be completely disjoint. So, by being a bit more careful, we can take $E_{t+1}$ to intersect one index in $U_t$ and still have things work, and now at each step we only expose 4 rather than 5 new indices.

\begin{proof}[Proof of Theorem~\ref{2better}]Fix $c\in \CC_{n;2}$.

For ease of notation, let $G_t$ be the graph on vertex set $[n]$, with $e \in [n]^{(2)}$ being an edge if there exists a twin $I,J$ of length $t$ such that $\{\max(I),\max(J)\} = e$. We shall use the following corollary of Lemma~\ref{extension}.
\begin{prp}\label{basic} Let $e,e'\in [n]^{(2)}$ be such that:
\begin{itemize}
    \item $e\in E(G_t)$;
    \item $\max(e)<\min(e')$;
    \item $e,e'$ are $c$-matchable.
\end{itemize}Then $e'\in E(G_{t+1})$.
\end{prp}

For $t\le n/4$, we shall find a triple $U_t \in [4t]^{(3)}$ such that $u\in E(G_t)$ for each $u\in U_t^{(2)}$ (i.e., $U_t$ induces a triangle in $E(G_t)$).

For $t = 1$, we can simply take $U_1 = \{1,2,3\}$, as $G_1$ is a clique.

Now suppose that we have some $U_t\in [n-4]^{(3)}$ which induces a triangle in $G_t$. We will find $U_{t+1}\in [n]^{(3)}$ with $\max(U_{t+1}) \le 4+\max(U_t)$, such that $U_{t+1}$ induces a triangle in $G_{t+1}$. By induction, it shall follow that $F_2(n) \ge \lfloor n/4\rfloor = n/4-O(1)$, as desired.

Let $ \{x<y<z\} = U_t$. Consider the discrete interval $Z = \{z,z+1,\dots,z+4\}$. We partition $Z$ into three sets,
\[A := \{l\in Z: c(\{x,l\}) =c(\{y,l\})= 1\}, \]
\[B := \{l\in Z: c(\{x,l\}) =c(\{y,l\})= 2\}.\]
\[S := \{l\in Z: c(\{x,l\}) \neq c(\{y,l\})\}.\]
\noindent By Proposition~\ref{basic}, it is clear that $G_{t+1}[A],G_{t+1}[B]$ are each cliques.

We will now finish by considering a few cases.

\textbf{Case 1} ($|S| = 0$): Here $|A\cup B| = 5$, thus we get $\max\{|A|,|B|\}\ge 3$ by pigeonhole. WLOG, assume $|A| \ge 3$. Since $G_{t+1}[A]$ is a clique, taking any $U_{t+1}\in A^{(3)}$ will induce a triangle.

\textbf{Case 2} ($|S| \in \{1,2\}$): Here $|A\cup B| \ge 3$, thus $\max\{|A|,|B|\}\ge 2$ by pigeonhole. WLOG, assume $|A|\ge 2$, and fix distinct $a,a'\in A$. Also, fix any $s\in S$ (which is possible as $|S|>0$). We shall take $U_{t+1} = \{a,a',s\}$.

Indeed, $\{a,a'\} \in E(G_{t+1})$ as $G_{t+1}[A]$ is a clique. Meanwhile, since $s\in S$, there exists $o\in \{x,y\}$ such that $c(\{o,s\}) = 1$. Thus taking $o^* = \{x,y\}\setminus \{o\}$, we have that $(o,s),(o^*,a)$ and $(o,s),(o^*,a')$ are $c$-matchings, implying $\{s,a\},\{s,a'\}\in E(G_{t+1})$ by Proposition~\ref{basic} as desired. Thus, we see $U_{t+1}$ induces a triangle.

\textbf{Case 3} ($|S| \ge 3$): It follows that we can pick distinct $s,s'\in S\setminus \{z\}$. Next, pick any $l\in Z\setminus \{s,s',z\}$. We take $U_{t+1} = \{s,s',l\}$.

We shall prove that for any distinct $s,l\in Z\setminus \{z\}$ with $s\in S$, that $\{s,l\}\in E(G_{t+1})$. This will clearly imply that $U_{t+1}$ induces a triangle.

Now, since $s\in S$, there exists $o\in \{x,y\}$ such that $c(\{o,s\}) = c(\{z,l\})$. Whence, $(o,z),(s,l)$ is a $c$-matching. Furthermore, since $s,l\in Z\setminus \{z\}$, we have that $\min\{s,l\}>z = \max\{o,z\}$, allowing us to invoke Proposition~\ref{basic} and deduce that $\{s,l\}\in E(G_{t+1})$ as desired.\end{proof}

\section{Some upper bounds}\label{upper bounds}

In this section, we improve the upper bounds for $F_r(n)$ for various $r$. Our proofs are reminisicant of the methods in a paper of Bukh and Guruswami \cite{guruswami} (see also their improved result with H\aa stad \cite{guruswami2}), where they construct large sets of strings no two of which have a ``long common subsequence''. Essentially, the idea will be to take a random string $x\in [R]^m$, where $R$ is some large number of colors so that we expect $x$ to only have very short twins. Then, we will encode $x$ as some edge-coloring $c_x:[n]^{(2)}\to [r]$, where we use fewer colors but still do not create particularly long twins.

We require the following bounds, which all follow from first moment considerations.

\begin{prp}\label{randomstrings}\cite[Theorem~4]{bukh}Let $x$ be a uniformly random $r$-ary string of length $n$. Asymptotically almost surely (as $n\to \infty)$, $x$ lacks any string-twins of length $O(n/\sqrt{r})$ (for some absolute constant independent of $r$).
\end{prp}\noindent We note that Proposition~\ref{randomstrings} is also a corollary of \cite[Theorem~3]{axenovich}.

\begin{defn}Given permutations $\pi,\pi'\in S_r$ (which we treat as bijections from $[r]\to [r]$), we define $\LCS(\pi,\pi')$ to be the largest $\ell$ such that there exists $A = \{a_1<\dots<a_\ell\},B=\{b_1<\dots<b_\ell\} \subset [r]$ (not necessarily disjoint) with $\pi(a_t) = \pi'(b_t)$ for $t=1,\dots,\ell$.

\end{defn}

\begin{prp}\label{randomperms}Let $\pi,\pi'\in S_r$ be chosen uniformly at random. We have $\P(\LCS(\pi,\pi') > 3\sqrt{r})\le  r^{-\omega(1)}$.
\begin{proof}We shall write $k := \lceil 3\sqrt{r}\rceil$.

For $A = \{a_1<\dots<a_\ell\},B = \{b_1<\dots<b_\ell\}\subset [r]$, let $Y_{A,B}$ be the indicator function of the event that $\pi(a_t) = \pi'(b_t)$ for each $t=1,\dots,\ell$. We then let
\[Y:= \sum_{\substack{A,B\subset [r]:\\ |A|=|B|=k}}Y_{A,B}.\]\edit{perhaps it is better to write $A,B\in [r]^{(k)}$?} It is clear that $\P(\LCS(\pi,\pi')\ge k)\le \E[Y]$.

Now for any choice of $A,B\in [r]^{(k)}$, we have \[\E[Y_{A,B}] = \frac{1}{\binom{r}{k} k!}\](here $1/\binom{r}{k}$ is the probability $\pi(A)= \pi(B)$, and $1/k!$ is the probability that our event holds conditioned on this). Meanwhile, there are only $\binom{r}{k}^2$ choices of sets $A,B\in [r]^{(k)}$. Whence, we see that
\[\E[Y] = \#(A,B\in [r]^{(k)})\frac{1}{\binom{r}{k}k!} = \binom{r}{k}\frac{1}{k!}\le \frac{r^k}{k!k!} = (1+o(1))\frac{1}{2\pi k}\left(\frac{e^2r}{k^2}\right)^k\](applying Stirling's approximation). By our choice of $k$, the LHS is at most $(e^2/9)^{-\sqrt{r}} = \exp(-\Omega(\sqrt{r}))$ for large $r$, meaning $\P(\LCS(\pi,\pi')\ge k)$ decays super-polynomially as desired.
\end{proof}
\end{prp}

Now, we start with the easier of our two upper bounds, which improves things for sufficiently large $r$.

\begin{repthm}{general upper} We have that
\[F_r(n)\le O(n/r).\]
\end{repthm}
\begin{proof}Assume $r$ is sufficiently large. Let $r^* = \lfloor r/2\rfloor$, $R = r^2$. We shall consider $n = mr'$.

Pick $x\in [r^*]^m,y\in [R]^m$ uniformly at random. Also, pick $\pi_1,\dots, \pi_R \in S_{r^*}$ uniformly at random.

By Propositions~\ref{randomstrings} and \ref{randomperms}, all of the following hold with positive probability (by a union bound):
\begin{itemize}
    \item $f^{\strng}(x) = O(m/\sqrt{r})$;
    \item $f^{\strng}(y) = O(m/r)$;
    \item we have $\LCS(\pi_i,\pi_j) = O(\sqrt{r})$ for each distinct $i,j \in [R]$.
\end{itemize}
We condition on such an outcome, and use this to construct our $c\in \CC_{n;r}$.

We define the maps $\Phi:[n]\to [m],\Psi:[n]\to [r^*]$ so that $\Phi(k) = \lceil k/r^*\rceil$ and $k = (\Phi(k)-1)r^*+\Psi(k)$.

Consider $k<k'\in [n]= [r^*m]$. We take
\[c(\{k,k'\}) = \begin{cases}x(\Phi(k))& \textrm{if } \Phi(k)<\Phi(k');\\
r^*+\pi_{y(\Phi(k))}(\Psi(k))&\textrm{otherwise}.
\end{cases}\]\noindent We call the first case of our definition the `global rule', and our second case the `local rule'.

Obviously, $c(\{k,k'\})$ always takes some value in $[2r^*]\subset [r]$, thus $c\in \CC_{n;r}$. We will prove that \begin{equation}\label{twinbound}f(c)\le m+2f^{\strng}(y)r+(2f^{\strng}(x)+1)(\max_{i\neq j}\{\LCS(\pi_i,\pi_j)\}+1),\end{equation}implying that $f(c)= O(m) = O(n/r)$ by our assumptions on $x,y,\pi_1,\dots,\pi_R$. Hence, we will be done after establishing Eq.~\ref{twinbound}.

Without further delay, we prove that $f(c)$ is small. Consider any twins $I= \{i_1<\dots<i_\ell\},J=\{j_1<\dots<j_\ell\}$ of $c$. For $t\in [\ell]$, we define $\phi_I(t) = \Phi(i_t)$ and $\psi_I(t) = \Psi(i_t)$. We similarly define $\phi_J,\psi_J$ by replacing `$i_t$' by `$j_t$'.

As $I,J$ are twins, we must have that $\phi_I(t) = \phi_I(t+1)$ if and only if $\phi_J(t) = \phi_J(t+1)$ (otherwise one of $c(\{i_t,i_{t+1}\}),c(\{j_t,j_{t+1}\})$ will belong to $[r^*]$ and the other will belong to $r^*+[r^*]$). 

It is clear that $\phi_I,\phi_J$ are non-decreasing, as they are obtained from apply the ceiling function to increasing sequences.

Next let $l =|\Phi(I)|$. We define $E_1, \dots, E_l\subset [\ell]$ so that $t\in E_h$ if and only if $\phi_I(t)$ is the $h$-th smallest element of $\Phi(I)$. We record that these sets:
\begin{itemize}
    \item partition $[\ell]$ (clear);
    \item are all intervals (as $\phi_I$ is non-decreasing);
    \item each have size at most $r^*$ (as $\Phi$ is a $r^*$-to-1 function).
\end{itemize}\noindent Also, let $A := \Phi(I) = \{a_1<\dots<a_l\} , B:= \Phi(J) = \{b_1<\dots<b_l\}$.

Let $H_1 := \{h\in [l]: a_h = b_h\},H_2 := \{h\in [l]\setminus H_1: y(a_h) = y(b_h)\},H_3 = [l]\setminus (H_1 \cup H_2)$. We will show that the indices contributed by each of the three parts is appropriately bounded.

\begin{prp}\label{H1bound}We have $|H_1|\le m$ and $|E_h| \le 1$ for $h\in H_1$.
\begin{proof} Since $H_1\subset [l]$ and $l =|\Phi(I)|$ we have $|H_1| \le |\im(\Phi)| = |[m]| = m$.

It remains to show that $h\in H_1$ implies that $|E_h| \le 1$.

Supposing otherwise, we'd have $\phi_I(t) = \phi_I(t+1) =a = \phi_J(t+1) = \phi_J(t)$ for some $t\in E_h$ and $a\in [m]$. Thus our local rule gives
\[c(\{i_t,i_{t+1}\}) = r^*+\pi_{y(a)}(\Psi(i_t))\]
while
\[c(\{j_t,j_{t+1}\}) = r^*+\pi_{y(a)}(\Psi(j_t)).\]Since $\pi_{y(a)}$ is injective, we should have $\Psi(i_t)=\Psi(j_t) = b$ (as $I,J$ are twins), which implies that $i_t = (a-1)r^*+b = j_t$. But then $I,J$ are not disjoint (and hence not twins), contradiction.\end{proof}
\end{prp}

\begin{prp}\label{H2bound} We have $|H_2|\le 2f^{\strng}(y)$ and $|E_h|\le r^*$ for $h\in H_2$.
\begin{proof}As previously noted, the bound $|E_h|\le r^*$ holds for all $h\in [l]$, thus the second part is trivial.

Now, construct a graph $G$ on vertex set $[m]$,
with $E(G) = \{a_hb_h:h\in H_2\}$. Note that $G$ will not have any loops, since $H_2\subset [l]\setminus H_1$ ensures that $a_h\neq b_h$ for $h\in H_2$.

Next, as $a_\cdot,b_\cdot$ are increaing sequences, we have that the connected components of $G$ are all paths. Thus there is is some matching (collection of disjoint edges $M\subset E(G)$ with $|M| \ge e(G)/2 = |H_2|/2$. 

Take $H_M:= \{h\in H_2: a_hb_h \in M\}$, which has size $|M|\ge |H_2|/2$. We finish by showing that $A_M:=\{a_h:h\in H_M\},B_M:=\{b_h:h\in H_M\}$ is a string-twin of $y$, which immediately rearranges to give the desired bound $|H_2|\le 2|H_M| = 2|A_M| \le 2f^{\strng}(y)$ (here $|H_M| =|A_M|$ follows for the fact that $h\neq h'\implies a_h\neq a_{h'}$).

Since $M$ is a matching, we have that $A_M,B_M$ are disjoint. For $t\in [|M|]$, write $\eta_t^*$ (respectively $\alpha_t^*,\beta_t^*$) for the $t$-th smallest element of $H_M$ (respectively $A_M,B_M$). Since $h<h'$ implies $a_h<a_{h'}$ and $b_h<b_{h'}$, we have that $\alpha_t = a_{\eta_t}$ and $\beta_t = b_{\eta_t}$. Whence, we have that $y(\alpha_t) = y(\beta_t)$ for $t=1,\dots,|M|$ (as $\eta_t\in H_2$ implies that $y(a_t) = y(b_t)$). 

So, we conclude that $A_M,B_M$ is a string twin of $y$, as desired.
\end{proof}
\end{prp}

\begin{prp}\label{H3bound}We have that $|H_3|\le 2f^{\strng}(x)+1$ and $|E_h|\le \max_{i\neq j}\{\LCS(\pi_i,\pi_j)\}+1$ for $h\in H_3$.
\begin{proof}We first bound $|E_h|$ assuming $h\in H_3$. As $h\not \in H_2$, we have that $y(a_h)\neq y(b_h)$. So the desired bound will follow from showing that $|E_h|\le \LCS(\pi_{y(a_h)},\pi_{y(b_h)})+1$ always holds. 
Let $\{i_1'<\dots<i_q'\} = \{i_t:t\in E_h\}, \{j_1' <\dots<j_q'\} = \{j_t:t\in E_h\}$. Also write $\tau = \pi_{y(a_h)},\sigma = \pi_{y(b_h)}$.

By our local rule, we must have that $\tau(\Psi(i_t')) = \sigma(\Psi(j_t'))$ for each $t\in [q-1]$. Whence, we see that $q-1\le \LCS(\tau,\sigma)$ (as $\psi_I|_{E_h}$ and $\psi_J|_{E_h}$ are both increasing).

It remains to bound $|H_3|$. Due to our global rule, for $t\in [l-1]$, we must have that $x(a_t)=x(b_t)$. Now we define $H' =[l-1]\setminus H_1$. Clearly $|H'|\le |H_3|+1$, as $H_3\subset [l]\setminus H_1$. 

Finally, by repeating the argument from Proposition~\ref{H2bound}, we can find $H_M'\subset H'$ with $|H_M'|\ge |H'|/2$ so that $\{a_h:h\in H_M'\},\{b_h:h\in H_M'\}$ is a string-twin of $x$. Some minor rearranging gives that $|H_3| \ge 2f^{sting}(x)+1$, completing the proof.\end{proof}
\end{prp}
Finally, it is clear to see that Eq.~\ref{twinbound} holds by combining Propositions~\ref{H1bound}, \ref{H2bound} and \ref{H3bound}, along with the fact that $\ell = \sum_{k=1}^3 \sum_{h\in H_k}|E_h|$. So we are done.
\end{proof}

We shall now prove the following upper bound, refuting a conjecture of Dudek, Grytczuk, and Ruzci\'nski.

\begin{repthm}{weak upper}There exists some absolute constant $\xi>0$ such that 
\[F_2(n)\le F_{weak}(n)\le (1/2-\xi+o(1))n.\]
\begin{rmk}We have made no effort to optimize the constant $\xi$ given by our argument.

For those curious, one may use $\ep = 1/400000$ and $\delta = 3/10000$ in the below arguments, which shows that $\xi> 3^{-2^{10001}}/400000$ is attainable. By using sharper estimates and being less sloppy, one could probably show $\xi >2^{-200}$. However, proving $\xi > 2^{-20}$ likely requires new ideas. \hide{See Section~\ref{conclusion} for more discussion.} \end{rmk}
\begin{proof}We will fix some sufficiently large integer $r$. For each letter $l\in [r]$, we assign the weight $w(l):= 3^l$. 

We shall consider a random $r$-ary string $x$ of length $m$. For $k=1,\dots,m$, let $L_k = \sum_{t=1}^k w(x_t)$. Also, set $L_0 =0$. 

Now given $x$, we form a permutation $\pi_x$ of length $L_m$. Specifically, for $k\in [m]$, we take
\[\pi_x|_{(L_{k-1},L_k]} = L_k,L_k-1,\dots, L_{k-1}+1\]
(in other words, $\pi|_x$ is the ``skew-sum'' $D_{w(x_1)} \ominus D_{w(x_2)} \dots \ominus D_{w(x_m)}$ where $D_l$ denote the decreasing permutation of length $l$).

For $i\in [n]$, let $k_i$ be the smallest $k$ such that $i\le L_k$. Also, let $E_k := [L_k]\setminus [L_{k-1}]$.

Now, $\pi_x$ corresponds to the 2-coloring $c\in \CC_{n;2}$ where for $i<j$, $c(\{i,j\}) =1$ if and only if $k_i = k_j$ (recall the construction given in Lemma~\ref{weak reduction}).

Now, $n\le 3^r m$, thus we wish to prove that for some $\ep>0$, that for all twins $I,J$ of $c$, we have $|[n]\setminus (I\cup J)| \ge \ep m$. Thus, given a pair of twins $I,J$, let $K_{I,J} := \{k: I\cup J\not \supset E_k\}$. We shall argue that a.a.s., $|K_{I,J}| \ge \ep m$.

Given a set $S = \{s_1<\dots<s_\ell\}$, we define a map $\phi = \phi_S:[\ell]\to [m]$ by setting $\phi(t) = k_{s_t}$ for $t =1,\dots,\ell$.

We make a few remarks. Obviously for any $S$, we have that $\phi_S$ is increasing (i.e., $\phi_S(t)\le \phi_S(t+1)$ for $t<|S|$). Also, if $I,J$ are twins, then we must have that $\phi_I(t)<\phi_I(t+1)$ if and only if $\phi_J(t)<\phi_J(t+1)$ (by definition of $c$). 

Now given a twin $I,J$, we define a graph $G = G_{I,J}$, where we will allow loops but won't care about the multiplicity of edges. Specifically, $G$ will have vertex set $[m]$ and edge set $\{\phi_I(t)\phi_J(t):t\in [\ell]\}$. We observe that the components of $G$ are all either singletons, loops, or paths $P$ with some vertex set $V = \{v_1<\dots<v_q\}$ with edge set $\{v_tv_{t+1}:t\in [q-1]\}$. Additionally, we observe that $G_{I,J}$ actually only depends on $\im \phi_I,\im \phi_J$. 

Thus, given $X,Y\subset [m]$, we write $G_{(X,Y)}$ to denote the unique graph $G_{I,J}$ where $I,J$ is a twin (of $c$) with $\im \phi_I = X,\im \phi_J=Y$ (or the empty graph if no such $I,J$ exist). Furthermore, let $\chi(X,Y)$ count the number of connected components $G_{(X,Y)}$.

\hide{Fix $K = K_{I,J} \subset [m]$. Also, take $d = r^2$ and let $B = B_{I,J}$ be the set of $v\in [m]$ such that there is a path of length at most $d$ from $v$ to some $k\in K$. We have that $|B|\le (2d+1)|K|$, thus it will suffice to argue that $|B| \ge \ep m$ with high probability.}

Given subsets $X,Y\subset [m]$ let $\EE_{X,Y}$ be the event that there exists a twin $I,J$ of $c$ such that $\im \phi_I = X,\im \phi_J = Y$ where $|K_{I\cup J}|<\ep m$. We will prove two bounds on $\P(\EE_{X,Y})$.
\begin{prp}\label{firstmom}Consider any $(X,Y)\in \UU$.

We have that 
\[\P(\EE_{X,Y}) < 2^{\chi(X,Y)}r^{-(\chi(X,Y)-\ep m)}.\]
\end{prp}
\begin{prp}\label{cherncor}Consider any $(X,Y)\in \UU$.

Suppose $m/3 - \chi(X,Y)= 8(\eta+\ep)m$ for some $\eta>0$, and that $r$ is even. Then we have that \[\P(\EE_{X,Y}) \le \exp(-288\eta^3m) < \exp(-\eta^3 m).\]

\end{prp}

We shall first show how this implies our desired result. We intend to do a union bound. Let $\UU$ be the set of pairs $(X,Y)$ of $X,Y\subset [m]$ with $|X|=|Y|$. Note that if $I,J$ is a twin, then $(\im \phi_I,\im \phi_J)\in \UU$. So, it suffices to prove
\[\sum_{(X,Y) \in \UU} \P(\EE_{X,Y}) <1.\]  

Pick $\delta>0$. We write $\UU = \UU_1 \cup \UU_2 \cup \UU_3$, where 
\[(X,Y)\in \begin{cases}\UU_1 &\textrm{if $|X\cup Y|\le (1-\ep)m$}\\
\UU_2& \textrm{if $\chi(X,Y)\ge  (\ep+\delta) m$} \\
\UU_3 &\textrm{otherwise.}\end{cases}\]

It is clear that if $\im\phi_I = X,\im\phi_J = Y$, that $[m]\setminus (X\cup Y) \subset K_{I, J}$, whence we see $\P(\EE_{X,Y})=0$ for all $(X,Y)\in \UU_1$.

Meanwhile, as $|\UU_2|\le |\UU| \le 2^{2m}$, we can apply Proposition~\ref{firstmom} to get \[\sum_{(X,Y)\in \UU_2} \P(\EE_{X,Y})\le 2^{3m} r^{- \delta m}< 1/2\] assuming $r\ge 2^{3\delta^{-1}+1}$.

Finally, for $(X,Y)\in \UU_3$,
we claim that $|X\cup Y|$ is large. Indeed, for any $(X,Y)\in \UU$ we have that \[|X|+|Y| = 2e(G_{(X,Y)}) \ge 2(|X\cup Y|-\chi(X,Y)) \](and equality holds on the RHS when $G_{(X,Y)}$ lacks loops). Thus as $(X,Y)\not \in \UU_1\cup \UU_2$, it follows that $|X|+|Y|\ge (2-4\ep-2\delta)m$. Take $\alpha := 4\ep+2\delta$.

Whence, we have that
\[|\UU_3| \le m^2\binom{m}{\alpha m}^2 \](since there are at most $(t+1)\binom{m}{t}$ sets $Z\subset [m]$ of size at most $t$ which can be the complement of a set having cardinality at least $n-t\ge n/2$). By Stirling's approximation, we get that 
\[|\UU_3| \le O(m^3 (\alpha^{-\alpha }(1-\alpha)^{\alpha-1})^{2m})<(\alpha^{-\alpha}(1-\alpha)^{-1})^{2m} \]for large $m$. Meanwhile, we have that $(X,Y)\in \UU_3$ implies that 
\[m/3-\chi(X,Y)-8\ep m  \ge (1-27\ep -3\delta )m/3 .\] Thus, taking $\eta= (1-27\ep-3\delta)/24$, we have that
\[\P(\EE_{X,Y}) <\exp(-\eta^3 m).\]

Thus
\[\sum_{(X,Y)\in \UU_3} \P(\EE_{X,Y}) \le (\alpha^{-\alpha}(1-\alpha)^{-1})^{2m}\exp(-\eta^3m) = \exp\left(m(2\alpha\ln\frac{1}{\alpha}+2\ln\frac{1}{1-\alpha} - \eta^3)\right).\]When $\ep,\delta \downarrow 0$, we have that $\alpha \downarrow 0$ and $\eta\uparrow 1/24$. As this happens, we have $2\alpha\ln \frac{1}{\alpha}+2\ln \frac{1}{1-\alpha}\downarrow 0$ and $\eta^3 \uparrow 1/13824>0$. Thus, for small $\ep,\delta>0$ we have that $\sum_{(X,Y)\in \UU_3}\P(\EE_{X,Y}) <1/2$ as desired.

It remains to establish our claims. Before doing so, we establish some notation.

Given $(X,Y)\in \UU$, we write $\kappa_{(X,Y)}$ to be the minimum of $|K_{I,J}|$ over all twins $I,J$ with $G_{I,J} = G_{(X,Y)}$. Clearly, $\EE_{X,Y}$ is the indicator function of the event that $\kappa_{(X,Y)}<\ep m$.

Next, given a graph $G$, we write $\CC(G)$ to denote the set of connected components of $G$. For any component $C\in \CC(G_{(X,Y)})$ (which is either a path with $e(C)\ge 0$ edges, or a single looped vertex), we write $\kappa_C$ for the minimum of $|K_{I,J} \cap V(C)|$ over all $I,J$ with $C\in \CC(G_{I,J})$. 

Inspecting our definitions, we get the following useful facts.
\begin{prp}\label{bycomponent}Consider $(X,Y)\in \UU$.

We have that
\[\kappa_{(X,Y)} \ge \sum_{C\in \CC(G_{(X,Y)})} \kappa_C.\]Furthermore, the set of random variables $\{\kappa_C:C\in \CC(G_{(X,Y)})\}$ are independent over the randomness of $c=c_x$.
\begin{proof}To see the first part, we note
\[\kappa_C = \min_{I,J: G_{I,J} = G_{(X,Y)}}\{|K_{I,J}|\}\]and for each such $I,J$ we have that
\[|K_{I,J}| = \sum_{C\in \CC(G_{I,J}) = \CC(G_{(X,Y)})}|K_{I,J} \cap V(C)| \ge \sum_{C\in \CC(G_{(X,Y)})} \kappa_C.\]

We now establish the second part. For each $C$, we observe that $\kappa_C$ depends only on $x|_{V(C)}$. Whence, independence immediately follows, as the components $\CC(G_{(X,Y)})$ are all vertex-disjoint.
\end{proof}
\end{prp}\begin{rmk}It is not too hard to see that in fact $\kappa_{(X,Y)} = \sum_{C\in \CC(G_{(X,Y)})} \kappa_C$, since we can optimize the intersection of $I,J$ with each component without any issues. We omit the details since we do not need an upper bound for $\kappa_{(X,Y)}$.\end{rmk}

Without further ado, we prove our claims.
\begin{proof}[Proof of Proposition~\ref{firstmom}]Let $\CC := \CC(G_{(X,Y)}),\chi:= \chi(X,Y) = |\CC|$. We want to upper bound $\P(\EE_{X,Y}) = \P(\kappa_{(X,Y)}<\ep m)$.

By Proposition~\ref{bycomponent}, we have that \[\kappa_{(X,Y)} \ge \sum_{C\in \CC} \kappa_C\]where the LHS is the sum of $\chi$ independent random variables.

We will show that for any component $C$, that \begin{equation}
    \P(\kappa_C=0)\le 1/r \label{probperfect}.
\end{equation} Assuming this, a quick union bound gives our desired result. Indeed,
\begin{align*}
    \P(\kappa_{(X,Y)}<\ep m) &\le \P(|\{C\in \CC:\kappa_C = 0\}| \ge \chi-\ep m)\\
    &\le \sum_{Z\subset \CC:|Z|\ge \chi-\ep m} \P(\kappa_C = 0\textrm{ for all } C\in Z)\\
    &= \sum_{Z\subset \CC:|Z|\ge \chi-\ep m} r^{-|Z|}\\
    &\le \#(Z\subset \CC:|Z|\ge \chi-\ep m) r^{-(\chi-\ep m)}\le 2^\chi r^{-(\chi-\ep m)}\\
\end{align*}(here the first inequality is because each $\kappa_C$ takes non-negative integer values).

It remains to prove that Inequality~\ref{probperfect} holds.

Fix any component $C\in \CC$. We consider two cases.

\underline{Case 1 ($|V(C)| = 1$):} Here, we show $\kappa_C = 1$ always holds, which is clearly more than sufficient. 
Let $V(C) = \{k\}$. Consider any twin $I,J$ with $C\in \CC(G_{I,J})$. We will argue that $k\in K_{I,J}$, which implies $|K_{I,J} \cap V(C)| = |V(C)| = 1$ as claimed.

Let $A = \phi_I^{-1}(k), B = \phi_J^{-1}(k)$. Observe $A=B$ must hold (otherwise, $G_{I,J}$ would have an edge from $k$ to some $k'\neq k$, contradicting that $C$ is a component of $G_{I,J}$). So, it follows that $|I\cap E_k| = |A| = |B| = |J\cap E_k|$. 

Meanwhile, as $I,J$ are twins, they are disjoint, and so $|(I\cup J) \cap E_k| = |I\cap E_k| +|J\cap E_k| = 2|A|$ is even. Meanwhile, $|E_k|$ is odd, thus $E_k\subset (I \cup J)$ cannot hold, implying $k\in K_{I,J}$ as desired. 

\underline{Case 2 ($|V(C)|>1$):} Let $V = \{v_1< \dots <v_q\}$. As noted before, we have that $E(C) = \{v_tv_{t+1}:t\in [q-1]\}$. We will prove that if $x(v_1)\neq x(v_q)$, then $k_C>0$. Consequently, this means $\P(k_C =0)\le 1/r$, as desired.

Consider any twin $I,J$ with $C\in \CC(G_{I,J})$. Also, suppose $x(v_1)\neq x(v_q)$.

For $t\in [q]$, let $a_t = |\phi_I^{-1}(v_t)|, b_t = |\phi_J^{-1}(v_t)|, e_t = |E_{v_t}| = 3^{x(v_t)}$. Also, WLOG assume $a_1\neq 0$. Since $E(C)\subset E(G_{I,J})$, we have that $a_t = b_{t+1}>0$ for $t\in [q-1]$. Furthermore, we must have that $b_1 = a_q = 0$, so that $C\in \CC(G_{I,J})$ (otherwise, $v_1$ or $v_q$ would have additional edges).

Now suppose for sake of contradiction that $\kappa_C=0$. Then, we must have that $a_t+b_t = e_t$ for all $t\in [q]$. 

We shall see inductively that $e_1\mid a_t$ yet $3e_1\nmid  a_t$ for $t\in [q-1]$. Consequently, since $a_q =0, b_q = a_{q-1}$, the condition $e_q = a_q+b_q$ will imply $e_1\mid e_q,3e_q\nmid e_q$, which can only happen if $e_1 = e_q$ (as they are both powers of 3), contradiction.

It remains to establish the desired divisibility conditions, namely that $e_1\mid a_t$ and $3e_1\nmid a_t$ for $t\in [q-1]$.

Recalling that we're assuming $e_t=a_t+b_t$ for all $t$, and that $b_1 = 0$, we must have that $a_1 = e_1$. Thus the divisibility condition is satisfied.

Now, suppose we have the divisibility condition for some $t\in [q-2]$. We shall deduce it also holds for $t+1$. 

By assumption, we have that $a_{t+1} = e_{t+1}-b_{t+1} = e_{t+1} -a_t$. Also, since $C\in G_{I,J}$ we have that $a_{t+1}>0$. In particular, we have that $e_{t+1}>a_t \ge e_1$ (here the last inequality is a consequence of our divisibility conditions and the fact that $a_t\ge 0$). But then, we have that $3e_1 \mid e_{t+1}$ (as $e_{t+1}$ is a power of 3). Thus, recalling $a_{t+1} = e_{t+1}-a_t$, we see that $a_{t+1} \equiv -a_t \pmod{3e_1}$, which tells us that $e_1\mid a_{t+1}$ and $3e_1\nmid a_{t+1}$ as desired.\end{proof}
\begin{proof}[Proof of Proposition~\ref{cherncor}]Let $\CC = \CC(G_{(X,Y)}),\chi = |\CC|$. Suppose $m/3-\chi= 8(\eta+\ep)m$ for some $\eta>0$.

We will show that for any component $C$, that $\kappa_C$ stochastically dominates\footnote{We write $\Bin(t,p)$ to denote the binomial distribution with $t$ trials each with parameter $p$.} $\Bin(\lfloor |V(C)|/3\rfloor,1/8)$. Whence, it will follow by Proposition~\ref{bycomponent} that
\[\kappa_{(X,Y)} \ge \sum_{C\in \CC}\kappa_C \]which stochastically dominates
\[\sum_{C\in \CC} \Bin(\lfloor |V(C)|/3\rfloor,1/8) = \Bin\left(\sum_{C\in \CC} \lfloor |V(C)|/3 \rfloor,1/8\right).\]Noting \[T:=\sum_{C\in \CC}\lfloor |V(C)|/3\rfloor \ge  m/3-|\CC| =  8(\eta+\ep)m ,\] we have that $T\in [8(\eta+\ep)m,m]$ and thus $\mu:= T/8 \ge (\eta+\ep)m$. Thus by the stochastic domination above and a Chernoff bound, we see that \[\P(\kappa_C<\ep m) \le \P\left(\Bin(T,1/8)<\left( 1-\frac{\eta}{\eta+\ep}\right) \mu\right) \le \exp(-(\eta/(\eta+\ep))^2\mu/2) \le \exp(-288\eta^3 m)\](the last inequality follows from the fact that $m/3-\chi =8(\eta+\ep)m$ for $\eta>0$ implies that $\eta+\ep\le 1/24$).

It remains to prove that $\kappa_C$ stochastically dominates $\Bin(\lfloor |V(C)|/3\rfloor,1/8)$.

If $|V(C)| = 1$ then there is nothing to prove, so we may suppose that $C$ is a path. Let $V(C) = \{v_1<\dots<v_q\}$.

For $t\in [q]$, let $X_t = x_{v_t}$. Let $\ell = \lfloor q/3\rfloor $. For $t\in [\ell]$, we define the interval $U_t = \{3t-2,3t-1,3t\}\subset [q]$. It is clear that for distinct $t,t'\in [\ell]$ that $U_t,U_{t'}$ are disjoint.

For $t\in [\ell]$, let $\EE_{t} = \EE_{C,t}$ be the event that $X_{3t-1}> r/2\ge \max\{X_{3t-2},X_{3t}\}$. It is not hard to see that the events $\{\EE_t:t\in [\ell]\}$ are independent, as they are determined by disjoint sets of random variables that are all independent of one another. Furthermore, as we have assumed $r$ is even, it is clear that $\P(\EE_t) = 1/8$ for all $t$.

Consequently, is $|\{t\in [\ell]: \EE_t\textrm{ holds}\}|$ is distributed identically to $\Bin(\ell,1/8) = \Bin(\lfloor |V(C)|/3\rfloor,1/8)$. So now it suffices to show \[\kappa_C\ge |\{t\in [\ell]:\EE_t\textrm{ holds}\}|.\]The above is immediate corollary of the following fact, completeing our proof.
\begin{clm}Consider $k_1<k_2<k_3\in [m]$. Suppose that $x_{k_2}>\max\{x_{k_1},x_{k_3}\}$.

Then for any twin $I,J$ with $\{k_1k_2,k_2k_3\}\subset E(G_{I,J})$, we have $k_2\in K_{I,J}$.
\begin{proof}Consider any such $I,J$. For $t=1,2,3$, set $a_t = |E_{k_t}\cap I|,b_t = |E_{k_t}\cap J|,e_t = |E_{k_t}|$. WLOG, we may assume that $a_1 = b_2, a_2 = b_3$. 

Now assuming $x_{k_2}>\max\{x_{k_1},x_{k_3}\}$, implying \begin{align*}
    e_2 &\ge \max\{3e_1,3e_3\}\\
    &> e_1+e_3 \\
    &\ge a_1+b_3 \\
    &= b_2+a_2.\\
\end{align*}Consequently, $E_{k_2}\setminus(I\cup J)$ must be non-empty, implying $k_2\in K_{I,J}$ as desired.
\end{proof}
\end{clm}\end{proof}\end{proof}
\end{repthm}

\section{Conclusion}\label{conclusion}

There are a number of questions one can ask about this regime. Here are some problems we believe to be interesting.

Let
\[\rho_r := \liminf_{n\to\infty}F_r(n)/n \](presumably, the limit is still well-defined if we replace `$\liminf$' with `$\lim$', but we don't think this matter is especially important). Theorems~\ref{main} and \ref{general upper} proved that
\[\Omega(1/r^2) \le \rho_r \le O(1/r).\] This naturally begs the question: 
\begin{que}How does $\rho_r$ grow (as $r\to \infty$)?

\end{que}\noindent We are inclined to believe that our upper bound is closer to the truth. Indeed, it already seems difficult to create a coloring $c\in \CC_{n;r}$ without a twin $(I,J)$ where:
\begin{itemize}
    \item $|i_t-j_t|\le r^2$ for all $t$;
    \item for every interval $E\subset [n]$ of length $10r^2$, we have $|I\cap E|\ge r/100$.
\end{itemize}(Here the second condition implies that $f(c)\ge n/1000r-O_r(1)$, while the first condition is just to restrict our attention to ``local strategies'' for finding twins.)

We remark that it is still open to determine the analogous limit for $r$-ary strings (the best lower bound is that there exist string-twins of length $\Omega(n/r^{2/3})$).

Here are some generalized problems we didn't know how to answer.

We say that a twin $I,J$ is \textit{non-crossing} if $\max I < \min J$. Define $f_{cross}(c)$ to be the longest non-crossing twin $I,J$ in $c$, and $F_{cross,r}(n):= \max_{c\in \CC_{r,n}}\{f_{cross}(c)\}$.
\begin{que}
Is $F_{cross,r}(n) =o(n)$ for some $r$?
\end{que}\noindent We believe the answer is yes even for $r=2$. Such a result is equivalent to proving that for every $k\ge 1,\ep >0$, that when $n$ is sufficiently large, there exists $c_1,\dots,c_k\in \CC_{n;2}$ such that for $i\neq j\in [k]$, we have that for any $I = \{i_1<\dots<i_\ell\},J = \{j_1<\dots<j_\ell\}\subset [n]$ with $|I| = |J| \ge \ep n$, that there is some $t$ where $c_i(\{i_t,i_{t+1}\})\neq c_j(\{j_t,j_{t+1}\})$.

\hide{Given $d\ge 1$, we say $I = \{i_1<\dots<i_\ell\},J= \{j_1<\dots<j_\ell\}$ is a $d$-twin if $c(\{i_t,i_{t+a}\}) = c(\{j_t,j_{t+a}\})$ for all $a\in [d],t\in [\ell-a]$. One can define $f_d(c)$ to be the length of the longest $d$-twin in $c$, and let $F_{r;d}(n) := \min_{c\in \CC_{n;r}}\{f_d(c)\}$.

Define $c_{r;d} = \liminf_{n\to\infty} \frac{F_{r;d}(n)}{n}$. It is not too hard to extend our methods to prove that $c_{r;d}>0$ always holds (the idea is to color hyperedges of a complete $(d+1)$-partite $(d+1)$-uniform hypergraph with $r^d$ colors, and }

Finally, rather than edge-coloring $K_n$, we could consider edge-colorings of $K_n^{(s)}$ (the complete $n$-vertex hypergraph of uniformity $s$). Here, one can say a twin is a pair of disjoint subsets $I = \{i_1<\dots<i_\ell\},J = \{j_1<\dots<j_\ell\}$ such that $c(\{i_t,i_{t+1},\dots,i_{t+s-1}\}) = c(\{j_t,j_{t+1},\dots,j_{t+s-1}\})$ for all $t\in [\ell-s+1]$. Define $F_{r}^{(s)}(n)$ in the natural way, and let $\rho_r^{(s)}:= \liminf_{n\to\infty} \frac{F_r^{(s)}(n)}{n}$. 

Our methods can be extended to prove that $\rho_r^{(s)} > 0$ for all $r,s\ge 1$. We only provide a sketch, as our bounds seem far from tight (roughly inverse tower-type). The idea is to generalize Lemma~\ref{matchable}, by considering $r$-edge-colorings of a complete $s$-partite hypergraph with parts $A,B_1,\dots,B_{s-1}$ with $|A| = r+1$ and taking $|A|\lll |B_1|\lll \dots \lll |B_{s-1}|$ (here $x\lll y$ means $y$ is sufficiently large with respect to $x$). We'd say $\vec{b} \in B_1\times \dots \times B_{s-1}$ is $i$-popular if there exists distinct $a,a'\in A$ such that $c(\{a\}\cup \{b_t:t=1,\dots,s-1\}) = c(\{a'\}\cup \{b_t:t=1,\dots,s-1\}) = i$. By pigeonhole, there exists some color $i\in [r]$ such that at least $1/r$ of the $(d-1)$-tuples $\vec{b}\in B_1\times\dots \times B_{s-1}$ are $i$-popular. Then, using something like dependent random choice we can ensure that there exists $B_1'\subset B_1,\dots, B_{s-1}'\subset B_{s-1}$ so that $|B_1'| = |A|,|B_2'|= |B_1|,\dots, |B_{s-1}'| = |B_{s-2}|$ with each $\vec{b} \in B_1'\times \dots \times B_{s-1}'$ being $i$-popular. One should then be able imitate the proof of Theorem~\ref{main} and show that $\rho_r^{(s)}$ is bounded away from zero. 

Already for $s=3$, it would be nice to understand how things work.
\begin{que}Do we have $\rho_r^{(3)} \le \exp(-\Omega(r))$?
\end{que}

Lastly, it could be nice to consider an intermediate problem. Given positive integers $r,s$, let $F_{r;s}(n)$ be the largest $\ell$ such that for every $c\in \CC_{n;r}$, there exists $I = \{i_1<\dots<i_\ell\},J=\{j_1<\dots<j_\ell\}$ where $c(\{i_t,i_{t+k}\}) = c(\{j_t,j_{t+k}\})$ for all $k\in [s-1],t\in [\ell-k]$. It is clear that $F_{r^{\binom{s}{2}}}^{(s)}(n)\le F_{r;s}(n)\le F_r(n)$, and it would be nice to explore whether we can obtain better bounds in this specialized setting. 

\begin{que}Do we have $\liminf_{n\to\infty}\frac{F_{r,3}(n)}{n} \ge r^{-\Omega(1)}$?
\end{que}


\begin{thebibliography}{}
\bibitem{axenovich} M. Axenovich, Y. Person, and S. Puzynina, \textit{A regularity lemma and twins in words,} in \textit{Journal of Combinatorial Theory, Series A} \textbf{120} (4) (2013), p. 733-743.


\bibitem{guruswami} B. Bukh and V. Guruswami, \textit{An improved bound on the fraction of correctable deletions,} in \textit{SODA} (2016), p. 1893-1901.

\bibitem{guruswami2} B. Bukh, V. Guruswami, and J. H\aa stad, \textit{An improved bound on the fraction of correctable deletions,} in \textit{IEEE Transactions on Information Theory} \textbf{63} (2017), p. 93-103.

\bibitem{bukh} B. Bukh and L. Zhou, \textit{Twins in words and long common subsequences in permutations,} in \textit{Israel Journal of Mathematics} \textbf{213} (2016), p. 183-209.


\bibitem{dudek} A. Dudek, J. Grytczuk, and A. Ruci\'nski, \textit{On weak twins and up-and-down subpermutations,} in \textit{Integers} \textbf{21A} (2021), Ron Graham Memorial Volume, Paper No. A10, 17 pp.






\end{thebibliography}
\end{document}